\documentclass{article}

\usepackage[english]{babel}
\usepackage{amssymb}
\usepackage{amsmath}
\usepackage{hyperref}
\usepackage{enumerate}
\usepackage{amsthm}
\usepackage{amsfonts}
\usepackage{color}
\usepackage[curve,color]{xypic}
\usepackage{diagbox}
\usepackage{rotating}

\usepackage{MnSymbol}

\DeclareMathSymbol{\mlq}{\mathord}{operators}{``}
\DeclareMathSymbol{\mrq}{\mathord}{operators}{`'}

\title{A holomorphic $(2,2)$-Theorem for Abelian varieties of CM-type}
\date{31.8.2026}
\author{Fritz H\"ormann\footnote{This work is funded by donations. If you found them interesting, please consider a contribution: \url{https://donorbox.org/fritz-hormann-independent-researcher-in-mathematics}.}}

\usepackage[numbers,sort&compress]{natbib}

\newtheorem{SATZ}{Theorem}[section]

\newtheorem{LEMMA}[SATZ]{Lemma}
\newtheorem{KEYLEMMA}[SATZ]{Key Lemma}
\newtheorem{DEF}[SATZ]{Definition}
\newtheorem{PROP}[SATZ]{Proposition}

\newtheorem{FRAGE}[SATZ]{Question}

\newtheorem{VERMUTUNG}[SATZ]{Conjecture}

\newtheorem{BEM}[SATZ]{Remark}

\newtheoremstyle{bare}        
  {}            
  {}            
  {\normalfont}                 
  {}                            
  {\mdseries\scshape}                   
  {}                            
  {.0em}                           
  {\thmnumber{#2}#1. \thmnote{\normalfont\textsc{(#3)}} } 

\theoremstyle{bare}
\newtheorem{PAR}[SATZ]{}

\newcommand{\comment}[1]{}

\newcommand{\Mat}[1]{ \left(\begin{matrix} #1 \end{matrix} \right) }

\newcommand{\R}{ \mathbb{R} }
\newcommand{\C}{ \mathbb{C} }
\newcommand{\Q}{ \mathbb{Q} }

\newcommand{\Z}{ \mathbb{Z} }

\newcommand{\OOO}{\text{\footnotesize$\mathcal{O}$}}
\newcommand{\OO}{ {\cal O} }

\DeclareMathOperator{\id}{id}

\DeclareMathOperator{\Hom}{Hom}

\DeclareMathOperator{\Tot}{Tot}

\DeclareMathOperator{\im}{im}
\DeclareMathOperator{\gr}{gr}

\DeclareFontFamily{U}{wncy}{}
    \DeclareFontShape{U}{wncy}{m}{n}{<->wncyr10}{}
    \DeclareSymbolFont{mcy}{U}{wncy}{m}{n}
    \DeclareMathSymbol{\Sha}{\mathord}{mcy}{"58} 

\begin{document}

\maketitle

\comment{
\section{Notation}

Let $L|K$ be a field extension, we denote by
\[ C^\bullet(L|K) \]
Amitsur's complex.
}

{\footnotesize  {\em 2020 Mathematics Subject Classification:} 14C30, 14K22  }

{\footnotesize  {\em Keywords:}  Hodge conjecture, Complex tori, Abelian varieties, Amitsur's cosimplicial ring, analytic Milnor $K$-sheaves }

\section*{Abstract}

For complex tori with algebraic normalized period matrix, in particular for Abelian varieties of CM-type, the map from the
second cohomology group of the analytic sheaf of the second Milnor $K$-groups
to the rational cohomology classes of Hodge type $(2,2)$ is surjective (up to denominators). 
 
\section{Introduction}

Let $X$ be a complex projective variety, $n$ a non-negative integer, $\mathrm{CH}^{n}(X)$ its Chow group of codimension $n$ cycles, 
and let $H^{2n}(X_{\mathrm{an}}, \C) = \bigoplus_{p+q=2n} H^{p,q}$ be the Hodge decomposition of its cohomology in degree $2n$. 
We have a commutative diagram 
\[ \xymatrix{  H^{n}(X, \mathcal{K}_n^M)_{\Q} \ar[d] \ar[r]^{\sim} &  \mathrm{CH}^{n}(X)_{\Q} \ar[d]^{\mathrm{cl}^n} \\
H^{n}(X_{\mathrm{an}}, \mathcal{K}_n^{M,\mathrm{an}})_{\Q}  \ar[r]^-{\mathrm{dlog}} & H^{n,n} \cap H^{2n}(X_{\mathrm{an}}, \Q)  &  }   \]
where the top horizontal isomorphism is ``Bloch's formula'' \cite[\S 7, Theorem 5.19]{Qui73} and the morphism denoted $\mathrm{dlog}$ is induced by 
 
 \[ (\frac{1}{2\pi i})^n \mathrm{dlog} \wedge \cdots \wedge \mathrm{dlog}: \Lambda^n_{\Z} \OO_X^* \to \Omega^n \] 
 and we have the famous
\begin{VERMUTUNG}[Hodge $(n,n)$-conjecture]
The map $\mathrm{cl}^n$ is surjective.
 \end{VERMUTUNG}
 The sheaf $\mathcal{K}_n^{M,\mathrm{an}}$ is the analytic sheafification of the pre-sheaf $U \mapsto K_n^M(\OOO_U)$ and figures for example in
 Bloch's articles \cite{Blo74, Blo77}\footnote{Caution: In \cite{Blo74}, Bloch uses the notation $\mathcal{K}_2^{M,\mathrm{an}}$ for a certain completion of the present sheaf. }.  In this article, we show the following result (cf.\@ page~\pageref{SATZ22PROOF} for the proof), which applies in particular to all Abelian varieties of CM-type: 
\begin{SATZ}[Holomorphic (2,2)-Theorem] \label{SATZ22}
Let $X$ be a complex torus with algebraic (normalized) period matrix. 
Then the morphism 
\[ \mathrm{dlog}: H^{2}(X, \mathcal{K}_2^{M,\mathrm{an}})_{\Q}  \to H^{2,2} \cap H^{4}(X, \Q)   \]
is surjective. 
\end{SATZ} 
What is thus missing for a proof of the Hodge conjecture for Abelian varieties of CM-type? Hazama shows in \cite{Haz02} that for this class of varieties it suffices to show that 
\[ \mathrm{dlog}: H^{2}(X, \mathcal{K}_2^M)_{\Q}   \to H^{2,2} \cap H^{4}(X_{\mathrm{an}}, \Q)    \] 
is surjective (i.e.\@ the Hodge $(2,2)$-conjecture). 
Using Serre's GAGA \cite{Ser56}, this would follow from (essentially) Theorem~\ref{SATZ22} and another, purely analytic statement, as follows: 

From the proof of Theorem~\ref{SATZ22}, given in this article, follows that the morphism from the group cohomology for $\Gamma \cong \Z^{2g}$
\[ H^{2}(\Gamma, \Lambda^2_{\Z} \OOO^*(\C^g))_{\Q} \to H^{2,2} \cap H^{4}(X_{\mathrm{an}}, \Q)   \]
is surjective, where $\C^g \to X_{\mathrm{an}}$ is a chosen uniformization.  
The Hodge $(2,2)$-conjecture --- and thus the full Hodge conjecture --- for these varieties would follow, if one could show that the image of the boundary map in non-Abelian group cohomology
\[ \delta: H^1(\Gamma, \mathrm{SL}(\OOO(\C^g))) \to H^{2}(\Gamma, K_2(\OOO(\C^g)))  \]
induced by the central extension of groups with $\Gamma$-action
\[ \xymatrix{ 1 \ar[r] &  K_2(\OOO(\C^g)) \ar[r] &  \mathrm{St}(\OOO(\C^g)) \ar[r] &  \mathrm{SL}(\OOO(\C^g)) \ar[r] &  1 } \]
contains (up to rational factors) the image of $H^{2}(\Gamma, \Lambda^2_{\Z} \OOO^*(\C^g))$ \footnote{I do not know whether it is reasonable to conjecture that this is the case.
Bloch's and Suslin's works seem to suggest that at least a possible first obstruction of the form $H^2(\Gamma, K_2(\OOO(\C^g))) \to H^4(\Gamma, K_3(\OOO(\C^g)))$ should vanish on this image. }.
Indeed, by GAGA, the morphism $\delta$ factors via $H^{2}(X, \mathcal{K}_2^M)$.

The proof of Theorem~\ref{SATZ22} is quite elementary and I emphasize that --- while Bloch in \cite[Section 8]{Blo74} relates the Steinberg relations with 
the Leibniz rule for differential forms passing to a certain completion of $\mathcal{K}_2^{M, \mathrm{an}}$  ---
in this proof neither the Leibniz rule nor the Steinberg relations play any role whatsoever. 
This is perhaps not surprising because the Hodge structures on complex tori are described purely in terms of multi-linear algebra.
It seems likely that it can be generalized to arbitrary $n$ and to arbitrary complex tori but so far in the proof some tricks are used
that work only with the above mentioned restrictions. 

The proof of Theorem~\ref{SATZ22} is as follows:
Let $X$ be a complex torus. 
For each subfield  $K \subseteq \C$, we have a resolution of the constant sheaf
\begin{equation} \label{eqres} \xymatrix{ 0 \ar[r] & K_X \ar[r] & \OOO_X \ar[r]^-{} & \Lambda^2_K \OOO_X \ar[r]^-{} & \cdots  } \end{equation}
where $\OOO_X$ is the structure sheaf, 
and a hypercohomology spectral sequence
\[ H^i(X, \Lambda^j_K \OOO_X) \Rightarrow H^{i+j}(X, K) \]
inducing a filtration $F^\bullet_{\OOO_X} H^{n}(X, K)$. 

A holomorphic $(n,n)$-theorem, i.e.\@ the surjectivity of  
\[ \mathrm{dlog}: H^{n}(X, \mathcal{K}_n^{M,\mathrm{an}})_{\Q}  \to H^{n,n} \cap H^{2n}(X, \Q)   \]
would be a consequence of the following three properties:
\begin{enumerate}
\item We have $F^i_{\OOO_X}  H^{2n}(X, \Q) = \im(H^i(X, \Lambda_{\Z}^{2n-i} \OOO_X^*)_{\Q} \to H^{2n}(X, \Q))$.
\item The induced filtration on $H^{2n}(X, \C)$ is the Hodge filtration.
\item The morphism $H^{2n}(X, \Q) \to H^{2n}(X, \C)$ coming from the functoriality of the resolution (\ref{eqres}) in $K$ is {\em strict} w.r.t.\@ the
filtrations $F^\bullet_{\OOO_X}$. 
\end{enumerate}
Property~1.\@  is quite obvious (see Proposition~\ref{PROPK}) and true for all compact complex manifolds.
Property~2.\@  is true for all complex tori and all $n$ by elementary Hodge theory (see Proposition~\ref{PROPL}). 
To settle Property~3.\@ we replace the double complexes that compute the hypercohomology of the resolutions (\ref{eqres}) 
by purely-linear-algebra objects of the following form: Let $\mathrm{d}: V \to W$ be a surjection with kernel $W'$ of $\C$-vector spaces with $\Q$-structure $V_{\Q}$ on $V$ which 
describes $H^1(X, \Q)$ as Hodge structure of weight 1.

Then we form cosimplicial dg-modules with (simplex-wise) two terms
\begin{gather}\label{eqp} \vcenter{ \xymatrix{ W'_\Q \otimes_{\Q} \Q[\Delta^{\bullet}] \otimes_{\Q} C^{\bullet}(L|\Q)  \ar[d] \\  V_{\Q} \otimes_{\Q} C^{\bullet}(L|\Q) }   }
 \to \vcenter{ \xymatrix{ W' \otimes_{L} L[\Delta^{\bullet}]   \ar[d] \\  V }  }
 \end{gather}
 (maps specified in Definition~\ref{KEYDEF})
where $L$ is a field of definition of $\mathrm{d}$ and $ C^{\bullet}(L|\Q)$  is  Amitsur's cosimplicial ring with $C^n(L|\Q) = L \otimes_{\Q} \cdots \otimes_{\Q} L$, and $\Delta^{\bullet}$ is the
tautological cosimplicial set with $\Delta^n = \{0, \dots, n\}$, and where $W_{\Q}'$ is any $\Q$-form of $W'$ (not assuming that $\mathrm{d}$ is defined over $\Q$!).
Applying $S^{2n}_{C^{\bullet}(L|\Q)}$ (simplex-wise the symmetric tensor product of dg-modules over $L \otimes_{\Q} \cdots \otimes_{\Q} L$), and $S^{2n}_L$, respectively,  and then the complex of unnormalized cochains, we obtain
a double complex whose total cohomology\footnote{always filtered by means of the columns filtration} is $H^{2n}(X, \Q)$ with the filtration from Property~1.\@ above, and $H^{2n}(X, L)$ with the Hodge filtration, respectively. 
Thus it suffices to prove that (\ref{eqp}) induces a strict map of filtered vector spaces on the total cohomology or, what amounts to the same, that the map 
\begin{gather*} E^{\infty}_{i,2n-i}( S^{2n}_{C^{\bullet}(L|\Q)}  \left( \vcenter{ \xymatrix{ W'_\Q \otimes_{\Q} \Q[\Delta^{\bullet}] \otimes_{\Q} C^{\bullet}(L|\Q)  \ar[d] \\  V_{\Q} \otimes_{\Q} C^{\bullet}(L|\Q) }   } \right) )
 \to E^{\infty}_{i,2n-i}  (S^{2n}_{L}   \left( \vcenter{ \xymatrix{ W' \otimes_{L} L[\Delta^{\bullet}]   \ar[d] \\  V  }   } \right))
 \end{gather*}
 induced by (\ref{eqp}), is injective for appropriate $i$. 
This is a statement purely in linear algebra which is trivially true for $i=0$ \footnote{because the morphism of double complexes is an isomorphism in the first column} and it is proven in Proposition~\ref{KEYPROP} for $i=1$ under the assumption that $L|\Q$ is Galois. Theorem~\ref{SATZ22} follows.

\section{Resolutions with multilinear algebra}\label{SECTRES}

\begin{PROP}\label{PROPRES}
Let $X$ be a topological space and $K$ a field of characteristic 0. Let $\OOO$ be a ring sheaf with embedding $K_X \hookrightarrow \OOO$. There are  resolutions with morphisms $P$ and $A$ between them
\[  \xymatrix{ 0 \ar[r] & K_X \ar[r] \ar@{=}[d] & \OOO \ar[r] \ar@{=}[d] & T^2_K \OOO  \ar[r]  \ar@<2pt>[d]^{P_1}  &   T^3_K \OOO  \ar[r]  \ar@<2pt>[d]^{P_2}  & \cdots   \\
 0 \ar[r] & K_X \ar[r]  & \OOO \ar[r] & \Lambda^2_K \OOO  \ar[r] \ar@<2pt>[u]^{A_1}  & \Lambda^3_K \OOO  \ar[r] \ar@<2pt>[u]^{A_2}  & \cdots }   \]
 such that $P A = \id$. In particular, the hypercohomology spectral sequences associated with the respective resolutions, induce the same filtration on $H^i(X, K)$ which can be described as
 \[ F^i_{\OOO}H^n(X, K) := \im \left( H^{n-i}(X, \Lambda^{i}_K \frac{\OOO}{K_X} ) \to H^n(X, K) \right). \]
\end{PROP}

\begin{proof}
Consider a $K$-vector space $V$ with a $K$-linear embedding $\iota: K \hookrightarrow V$. 
We omit $\iota$ from the notation and denote for example $\iota(1)$ simply by $1$. We may form a coaugmented co-semisimplicial $K$-vector space
\begin{equation} \label{eqkv} \xymatrix{ K \ar@{-->}[r] & V \ar@<2pt>[r] \ar@<-2pt>[r]  & T^{2}_K V \ar@<-4pt>[r] \ar[r] \ar@<4pt>[r]   \ar[r] & \cdots  }\end{equation}
with coface maps given by
\[ \delta_i:  v_0 \otimes \cdots \otimes v_n \mapsto  v_0 \otimes \cdots \otimes v_{i-1} \otimes 1 \otimes v_{i} \otimes \cdots \otimes v_{n}.  \]
Its (unnormalized) cochain complex is exact with contracting homotopy
\[ h_n: v_0 \otimes \cdots \otimes v_n \mapsto \varepsilon(v_0) v_1 \otimes \cdots \otimes v_n  \]
where $\varepsilon: V \to K$ is any $K$-linear splitting of $\iota$,
and yields thus a resolution of $K$
\[ \xymatrix{ 0 \ar[r] &K \ar[r] & V \ar[r] & T^{2}_K V  \ar[r] & \cdots.  }\]
If $V$ is a $K$-algebra then (\ref{eqkv}) extends to a coaugmented cosimplicial $K$-vector space. 

There are morphisms of complexes
\[ \xymatrix{ K \ar[r] \ar@{=}[d] & V \ar[r] \ar@{=}[d] & \Lambda^2_K V  \ar[r] \ar@<2pt>[d]^{P_1} &  \Lambda^3_K V  \ar[r] \ar@<2pt>[d]^{P_2} &  \cdots   \\
 K \ar[r]  & V \ar[r]  & T^2_K V^{\otimes 2}  \ar[r] \ar@<2pt>[u]^{A_1} & T^3_K V  \ar[r] \ar@<2pt>[u]^{A_2} & \cdots  } \]
where $P_n$ is the quotient map followed by multiplication with $\frac{1}{(n+1)!}$  and $A_n$ is the map (antisymmetrization)
\begin{equation} \label{eqanti} A_n: v_0 \wedge \cdots \wedge v_n \mapsto \sum_{\sigma \in S_{n+1}} (-1)^{|\sigma|} v_{\sigma(0)} \otimes \cdots \otimes v_{\sigma(n)},
\end{equation}
such that $A P = \id$. The first row with differential
\begin{equation}\label{eq1} v_0 \wedge \cdots \wedge v_n \mapsto  1 \wedge v_0 \wedge \cdots \wedge v_{n}  \end{equation}
 is therefore also exact.
Splitting the first row into short exact sequences we get 
\[ \xymatrix{ 0 \ar[r] & \Lambda^i (V/K) \ar[r]  & \Lambda^{i+1} V  \ar[r]  & \Lambda^{i+1} (V/K)    \ar[r]  & 0 }  \]
where the inclusion is induced by (\ref{eq1}) and the projection is the quotient map. Its exactness can also be seen directly from
choosing a splitting $V = K \oplus W$ and using $\Lambda^{n+1} V \cong \Lambda^{n+1} W \oplus K \otimes \Lambda^{n} W$.
We have thus a resolution
\[ \xymatrix{ 0 \ar[r] & \Lambda^i (V/K) \ar[r]  & \Lambda^{i+1} V  \ar[r]  & \Lambda^{i+2} V   \ar[r]  & \cdots }  \]

The statement follows by applying the above to $V = \OOO(U)$ for all open $U \subseteq X$ and sheafifying.  
\end{proof}

\begin{PROP}\label{PROPL}
Let $X$ be a compact complex manifold with structure sheaf $\OOO_X$. We have
a morphism of resolutions
\[  \xymatrix{ 0 \ar[r] & \C_X \ar[r] \ar@{=}[d] & \OOO_X \ar@{=}[d] \ar[r] & T^2_{\C} \OOO_X \ar[d] \ar[r] & \cdots   \\
 0 \ar[r] & \C_X \ar[r]  & \OOO_X \ar[r] & \Omega^1_{\C}  \ar[r] & \cdots }   \]
 where the first row comes from Proposition~\ref{PROPRES} and the second row is the holomorphic de Rham complex. 
 If $X$ is a complex torus then this morphism induces a filtered isomorphism. In other words, the filtration $F^{\bullet}_{\OOO_X}H^n(X, \C)$ is the Hodge filtration. 
\end{PROP}
\begin{proof}
The morphism is given by
 \[ v_0 \otimes \cdots \otimes v_n \mapsto v_0 \mathrm{d} v_1 \wedge \cdots \wedge \mathrm{d} v_n . \]
 The composition with the antisymmetrization $A$ (\ref{eqanti})  obviously induces just the wedge product on the successive kernels in the resolutions 
 \[ \Lambda^i_{\C} \Omega^1_c \to \Omega^i_c. \]
For all non-negative integers $k_1 + \cdots + k_i = n-i$ we 
have a commutative diagram given by the cup product
\[ \xymatrix{ H^{k_1}(X, \Omega^1_c) \otimes \cdots \otimes H^{k_i}(X, \Omega^1_c) \ar[r] \ar@{^{(}->}[d] & H^{n-i}(X,  \Lambda^i_{\C} \Omega^1_c) \ar[rd] \ar[r] & H^{n-i}(X,  \Omega^i_c) \ar@{^{(}->}[d] \\
H^{k_1+1}(X, \C) \otimes \cdots \otimes H^{k_i+1}(X, \C) \ar[rr] & &  H^{n}(X,  \C).
 } \]
By Hodge theory for complex tori \cite[Chapter I]{Mum70} the top morphism (composition) is surjective. Hence the maps 
\[H^{n-i}(X,  \Lambda^i_{\C} \Omega^1_c) \to H^{n}(X,  \C) \]
and
\[H^{n-i}(X,   \Omega^i_c) \to H^{n}(X,  \C) \]
have the same image. 
\end{proof}

\begin{PROP}\label{PROPK}
Let $X$ be a compact complex manifold. We have
\[ F^i_{\OOO_X} H^n(X, \Q) = \im \left( H^{n-i}(X, \Lambda^i_{\Z} \OOO_X^*)_{\Q} \to H^n(X, \Q) \right)   \]
\end{PROP}
\begin{proof}
The discussion in the proof of Proposition~\ref{PROPRES} shows that
\[ F^i_{\OOO_X}  H^n(X, \Q) = \im \left( H^{n-i}(X, \Lambda^i_{\Q} \frac{ \OOO_X }{\Q_X})  \to H^n(X, \Q) \right)   \]
and we have 
\[ \Lambda^i_{\Q} \frac{ \OOO_X }{\Q_X} = ( \Lambda^i_{\Z} \OOO_X^* ) \otimes_{\Z} \Q \]
using that $\OOO_X^*  \cong \OOO_X / \Z_X$ via $\frac{\mathrm{log}}{2\pi i}$ (exponential sequence). Furthermore, we have
\[ H^{n-i}(X, ( \Lambda^i_{\Z} \OOO_X^* ) \otimes_{\Z} \Q ) \cong  H^{n-i}(X,  \Lambda^i_{\Z} \OOO_X^* ) \otimes_{\Z} \Q \]
because $X$ is compact and $- \otimes_{\Z} \Q$ commutes with finite products. 
\end{proof}

\begin{LEMMA}
The diagram
\[ \xymatrix{ H^n(X, \Lambda^n_{\Z} \OOO_X^*)_{\Q} \ar[drr] \ar[rr]^{\Lambda^n (\frac{1}{2 \pi i}\mathrm{dlog})} & & H^{n}(X, \Omega^n_X) \ar[d] \\  
 & & H^{2n}(X, \C) } \]
 where the vertical morphisms are the morphisms from the resolution and the morphism from Hodge theory, respectively, commutes.
\end{LEMMA}
\begin{proof}
Obviously the diagram
\[ \xymatrix{ H^n(X, \Lambda^n_{\Z} \OOO_X^*)_{\Q} \ar[drr] \ar[rr]^{\Lambda^n (\frac{1}{2 \pi i}\mathrm{dlog})} & & H^{n}(X, \Lambda^n_{\C} \Omega^1_{c,X}) \ar[d] \\  
 & & H^{2n}(X, \C) } \]
  where the vertical morphisms are induced by the resolutions, commutes
 because it comes from a morphism of resolutions (functoriality in $K$). In the proof of Proposition~\ref{PROPL} we have seen that the right hand side map factors as
 \[ H^{n}(X, \Lambda^n_{\C} \Omega^1_{c,X}) \to H^{n}(X,  \Omega^n_{c,X}) \to H^{2n}(X, \C).  \]
 However, its image lies in the subspace $H^{n}(X,  \Omega^n_{X}) \cong H^{n}(X,  \Omega^n_{c,X})  \cap \overline{H^{n}(X,  \Omega^n_{c,X})}$ by Hodge theory because the image is real. 
\end{proof}

\section{Amitsur's cosimplicial ring}

This section contains a brief discussion of Amitsur's cosimplicial ring \cite{Ami59}. 

\begin{DEF} \label{DEFAMITSUR}
Let $L|K$ be a field extension. 
{\bf Amitsur's coaugmented cosimplicial ring} $C^{\bullet}(L|K)$ is defined as
\[ \xymatrix{ K  \ar@{-->}[r] & L \ar@<3pt>[r] \ar@<-3pt>[r]  \ar@{<.}[r]  & L \otimes_K L  \ar@<-6pt>[r] \ar[r] \ar@<6pt>[r]  \ar@<-3pt>@{<.}[r]  \ar@<3pt>@{<.}[r]  \ar[r] & \cdots   }\]
with cofaces and codegeneracies given by
\begin{eqnarray*}
 \delta_i: l_0 \otimes \cdots \otimes l_n &\mapsto& l_0 \otimes \cdots \otimes l_{i-1} \otimes 1 \otimes   l_i \otimes \cdots \otimes l_{n},   \\
 s_i: l_0 \otimes \cdots \otimes l_n &\mapsto& l_0 \otimes \cdots \otimes l_{i-1} \otimes l_{i}  l_{i+1} \otimes l_{i+2} \otimes \cdots \otimes l_{n}  . 
\end{eqnarray*}
\end{DEF} 

The associated unnormalized cochain complex, the Amitsur complex, is acyclic. This follows by taking  $L' = K$ and $X^{\bullet} = K$ (constant cosimplicial $K$-vector space) in the following more general:

\begin{PROP}\label{PROPAMITSUR}
Let $L|L'|K$ be field extensions. 
For any cosimplicial $C^{\bullet}(L'|K)$-module $X^{\bullet}$ we have a homotopy equivalence (of the associated unnormalized cochain complexes)
\[ \xymatrix{ X^{\bullet} \ar@<3pt>[rr]^-{\iota}  & &  \ar@<3pt>[ll]^-{\id \otimes \chi}  X^{\bullet} \otimes_{C^{\bullet}(L'|K)} C^{\bullet}(L|K). }  \]
\end{PROP}
\begin{proof}
There is a section $\chi$ of the inclusion of the corresponding unnormalized cochain complexes
\[ \iota: C^{\bullet}(L'|K)  \to C^{\bullet}(L|K) \]
which is point-wise $C^{n}(L'|K)$-linear (it is compatible with the coface maps, but not compatible with the codegeneracies). It is given by
\[ \chi_n: m_0 \otimes \cdots \otimes m_n \mapsto \chi(m_0) \otimes \cdots \otimes \chi(m_n) \]
where $\chi: L \to L'$ is any $L'$-linear section of the inclusion. 

Define the operator 
\begin{equation*} 
\kappa: X_n \otimes_{C^{n}(L'|K)} C^{n}(L|K) \to X_{n-1} \otimes_{C^{n-1}(L'|K)} C^{n-1}(L|K) 
\end{equation*}
given by mapping 
$ x \otimes m_0 \otimes \cdots \otimes m_n$ to  \[  \sum_{i=1}^n (-1)^i s_{i-1}(x) \otimes \chi(m_0) \otimes \cdots \otimes \chi(m_{i-2})  \otimes \chi(m_{i-1}) m_i \otimes m_{i+1} \otimes  \cdots  \otimes m_n.   \]
Note that this is well-defined, i.e.\@ it respects the relations of $\otimes_{C^{n}(L'|K)}$ because $\chi$ is $L'$-linear and $X^{\bullet}$ is a cosimplicial $C^{\bullet}(L'|K)$-module.
A short calculation shows:  
\[ \id - \chi \iota = \mathrm{d} \kappa + \kappa \mathrm{d}. 	\qedhere \]
\end{proof}

Let $L|K$ be a Galois extension with Galois group $G$. Recall the cosimplicial $K$-algebra that computes group cohomology with values in $L$:
\[ \xymatrix{ 
L \ar@<3pt>[r] \ar@<-3pt>[r] &\ar@{.>}[l]  \Hom(G, L) \ar@<6pt>[r] \ar[r] \ar@<-6pt>[r]  & \ar@<3pt>@{.>}[l] \ar@<-3pt>@{.>}[l]  \Hom(G^2, L) 
\ar@<-3pt>[r]  \ar@<-9pt>[r] \ar@<3pt>[r] \ar@<9pt>[r]  & \ar@<6pt>@{.>}[l] \ar@{.>}[l] \ar@<-6pt>@{.>}[l]  \cdots 
} \]
with cofaces and codegeneracies given by 
\begin{eqnarray*} (\delta_i a)(\sigma_1, \dots, \sigma_{n+1}) &=& \begin{cases} {}^{\sigma_1} a(\sigma_2, \dots, \sigma_{n+1}) & i = 0, \\
 a(\sigma_1, \dots, \sigma_{i} \sigma_{i+1}, \dots, \sigma_{n+1})  &   i=1, \dots, n, \\
  a(\sigma_1, \dots, \sigma_{n}) & i = n+1, 
  \end{cases}  \\
 (s_i a)(\sigma_1, \dots, \sigma_{n-1}) &=&  a(\sigma_1, \dots, \sigma_{i}, 1, \sigma_{i+1}, \dots, \sigma_{n-1}) .  
\end{eqnarray*}

\begin{PROP}\label{PROPGALOIS} Let $L|K$ be a Galois extension with Galois group $G$.
There is an isomorphism of Amitsur's cosimplicial $K$-algebra
with the usual cosimplicial $K$-algebra computing the group cohomology of $G$ with values in $L$:
\[ \xymatrix{ 
L  \ar@<3pt>[r] \ar@<-3pt>[r] \ar@{=}[d]  & \ar@{.>}[l] L \otimes_K L    \ar@<6pt>[r] \ar[r] \ar@<-6pt>[r]   \ar[d]^{\sim}  & \ar@<3pt>@{.>}[l] \ar@<-3pt>@{.>}[l]  L \otimes_K L \otimes_K L  \ar[d]^{\sim} \ar@<-3pt>[r]  \ar@<-9pt>[r] \ar@<3pt>[r] \ar@<9pt>[r]  & \ar@<6pt>@{.>}[l] \ar@{.>}[l] \ar@<-6pt>@{.>}[l]  \cdots 
  \\
L \ar@<3pt>[r] \ar@<-3pt>[r] &\ar@{.>}[l]  \Hom(G, L) \ar@<6pt>[r] \ar[r] \ar@<-6pt>[r]  & \ar@<3pt>@{.>}[l] \ar@<-3pt>@{.>}[l]  \Hom(G^2, L)  
\ar@<-3pt>[r]  \ar@<-9pt>[r] \ar@<3pt>[r] \ar@<9pt>[r]  & \ar@<6pt>@{.>}[l] \ar@{.>}[l] \ar@<-6pt>@{.>}[l]  \cdots 
} \]
It maps  $l_0 \otimes \cdots \otimes l_n$ to 
\[ (\sigma_1, \dots, \sigma_n) \mapsto   l_0 {}^{\sigma_1} l_1 {}^{\sigma_1 \sigma_2 } l_2 \  \cdots \  {}^{\sigma_1 \cdots \sigma_n} l_n.  \]
\end{PROP}
\begin{proof}
One checks the compatibility with the coface and codegeneracy maps. It is injective, hence an isomorphism, because of the linear independence of characters. 
\end{proof}

The $n+1$ different $L$-module structures on $C^n(L|K)$ translate to the structures on
$\Hom(G^{n}, L)$
given by
\[ (\lambda \circ_i a)(\sigma_1, \dots, \sigma_{n}) = ({}^{\sigma_1 \cdots \sigma_i} \lambda) a(\sigma_1, \dots, \sigma_{n}) \]
(for $i=0$ this is thus the canonical $L$-module structure). 

The multiplication map $L \otimes_K \cdots \otimes_K L \to L$ translates to evaluation at $(1, \dots, 1)$. 
This implies the existence of a splitting
\[ L \to L \otimes_K \cdots \otimes_K L \]
given by the delta function at $(1, \dots, 1)$. This splitting is not a morphism of cosimplicial objects but 
$L$-linear w.r.t.\@ {\em any} of the $n+1$ natural $L$-module structures on $L\otimes_K \cdots \otimes_K L$.

\begin{PAR}\label{PARSYMMETRIC}
The functor ``symmetric algebra'' from dg-modules over $L$ to dg-algebras over $L$ is a left adjoint and thus transforms coproducts (direct sums) into coproducts (tensor products). Therefore we have, for any dg-module $B$ over $L$, isomorphisms
of dg-algebras
\[ S^{\bullet}_L (B^{n+1}) \cong T^{n+1}_L (S^{\bullet}_L B) .   \]
The isomorphism is in tensor degree one explicitly given by mapping: 
\[ (v_0, \dots, v_n) \mapsto \sum_i 1 \otimes \cdots \otimes 1 \otimes v_i \otimes 1 \otimes \cdots \otimes 1,    \]
with $v_i$ in the corresponding summand at position $i$. A little modification works over $K$:
\end{PAR}

\begin{PROP}\label{PROPSYMMETRIC}
Let $B$ be a dg-module over $L$. We have an isomorphism of dg-algebras
\[ S^{\bullet}_{C^n(L|K)} \left( \bigoplus_{i=0}^n  B \otimes_{L,i} C^n(L|K) \right) \cong T_K^{n+1} (S^{\bullet}_L B)   \]
via the obvious isomorphism $C^n(L|K) \cong T_K^{n+1} (S^0_L B)$ in degree 0 and via
\[ v \otimes l_1 \otimes \cdots \otimes l_n \mapsto l_1 \otimes \cdots \otimes l_{i-1} \otimes l_i v \otimes l_{i+1} \otimes \cdots \otimes l_n,    \]
 in the $i$-th direct summand. 
 The corresponding cosimplicial dg-algebras are modules under Amitsur's cosimplicial algebra $C^{\bullet}(L|K)$.
\end{PROP}
\begin{proof}Straightforward. \end{proof}

\section{Linear algebra}

Section~\ref{SECTRES} shows that, for a complex torus $X$, the surjectivity of
\[ H^n( X, \Lambda^n_{\Z} \OOO^*_X )_{\Q} \to H^{n,n} \cap H^{2n}(X, \Q)     \]
is equivalent to   
\[  H^{2n}(X, \Q) \to  H^{2n}(X, \C) \]
being strict for the $n$-th filtration step, w.r.t.\@ the filtrations induced by the resolutions
\[ \xymatrix{ 0 \ar[r] &  K_X \ar[r] &  \OOO_X \ar[r] &  T^2_K \OOO_X \ar[r] &   T^3_K \OOO_X \ar[r] &  \cdots }  \]
for $K=\Q$ and $K=\C$.
In this section we will describe these filtrations purely in terms of linear algebra.

\begin{PAR}\label{PARSETTING}
Let  $L|K$ be a field extension in characteristic 0, let
\[ \xymatrix{  0  \ar[r] & W'  \ar[r]^{\iota} & V  \ar[r]^{\mathrm{d}} &  W \ar[r] &  0 } \]
be an exact sequence of $L$-vector spaces,  
and let $V_K$ be a $K$-structure on $V$ (i.e.\@ a $K$-vector space with fixed isomorphism $V \cong V_K \otimes_K L$).
This datum can be seen as a filtration on $V$ 
\begin{equation}\label{eqhodge1} 
 F^0(V) = V \quad F^1(V) = W' \quad F^2(V) = 0 
\end{equation}
and it induces a filtration on $\Lambda^n_L V$ 
\begin{equation}\label{eqhodge} 
 F^i(\Lambda^{n}_L V) = \Lambda^i_L W' \otimes_L \Lambda^{n-i}_L V  
\end{equation}
with 
\[ \gr^i(\Lambda^n_L V) \cong \Lambda^i_L W' \otimes_L \Lambda^{n-i}_L W. \]
\end{PAR}

\begin{PAR}
For $L = \C$ and $K=\Q$ we almost get the definition of a Hodge structure of weight 1 
(up to the condition $V = W' \oplus \overline{W'}$). If $X$ is
a complex torus with Hodge structure $V = H^1(X, \C)$, $W = H^1(X, \OOO_X)$, $W' = H^0(X, \Omega^1_X)$, and $V_{\Q} =  H^1(X, \Q)$, then 
it is well-known \cite[Chapter I]{Mum70} that (\ref{eqhodge}) describes the Hodge filtration on
\[ H^n(X, \Q) \cong \Lambda^n_{\Q} H^1(X, \Q). \]
\end{PAR}

\begin{PAR}\label{PARCANCOSIMP}
Let $K[\Delta^{\bullet}]$ be the cosimplicial free $K$-vector space generated by the canonical cosimplicial set $\Delta^{\bullet}$:
\[ \xymatrix{   \Delta_0 \ar@<2pt>[r] \ar@<-2pt>[r] & \Delta_1  \ar@<-4pt>[r] \ar[r] \ar@<4pt>[r]  \ar[r] & \cdots   }\]
with $\Delta_i = \{0, \dots, i\}$. It is augmented via summation over the coefficients. Define $K[\Delta^{\bullet}]_0$ by means of the exact sequence:
\[ \xymatrix{ 0 \ar[r] &  K[\Delta^{\bullet}]_0 \ar[r] &   K[\Delta^{\bullet}] \ar[r] & K  \ar[r] & 0. } \]
\end{PAR}

\begin{LEMMA}\label{LEMMACONTRACTDELTA}
Consider a cosimplicial $K$-vector space $X^{\bullet}$.
The unnormalized cochain complex of $K[\Delta^{\bullet}] \otimes_K X^{\bullet}$ is acyclic. A contracting homotopy is explicitly given by
\[ \kappa_n : \Mat{x_0 \\ \vdots \\ x_n} \mapsto \Mat{s_1x_1 \\ - s_2x_2 \\ \vdots \\ (-1)^{n-1}s_{n}x_n} \]
\end{LEMMA}
\begin{proof}
Left to the reader. 
\end{proof}

\begin{DEF}\label{KEYDEF}
We define the following cosimplicial dg-module over $C^{\bullet}(L|K)$ (Amitsur's cosimplicial ring, cf.\@ Definition~\ref{DEFAMITSUR}):
\[  A^{\bullet}(L|K) := \left( \vcenter{ \xymatrix{  
W'_{K} \otimes_{K} K[\Delta^{\bullet}] \otimes_{K} C^{\bullet}(L|K)  \ar[d]^{\mathrm{d}_A} \\
 V_K \otimes_K C^{\bullet}(L|K) 
}} \right). \]
The differential $\mathrm{d}_A$ is given on the $i$-th summand by $\iota \otimes_{L,i} C^{n}(L|K)$ (cf.\@ \ref{PARSETTING}) using the {\em canonical} identifications
\begin{eqnarray*}  
W_K' \otimes_K [i] \otimes_K C^{n}(L|K) \cong W_K'  \otimes_K L \otimes_{L,i} C^{n}(L|K) &\cong& W' \otimes_{L,i} C^{n}(L|K),  \\
 V_K \otimes_K C^{n}(L|K) \cong V_K \otimes_K L \otimes_{L,i} C^{n}(L|K)  &\cong& V \otimes_{L,i} C^{n}(L|K).  
\end{eqnarray*}
\end{DEF}
One checks that this is compatible with the coface and codegeneracy maps. 
For $K=L$, we get just
\[ A^{\bullet}(L):= A^{\bullet}(L|L)  =  \left( \vcenter{ \xymatrix{  W' \otimes_L L[\Delta^{\bullet}] \ar[d] \\ V } } \right) \]
with $W'$ and $V$ considered as constant cosimplicial $L$-vector spaces and the differential being induced by the inclusion and the
augmentation $L[\Delta^{\bullet}] \to L$. 

\begin{PAR}We will be interested in the following cosimplicial dg-modules over $K$: 
\[ S^m_{C^{\bullet}(L|K)} A^{\bullet}(L|K). \]
The notation means the cosimplicial dg-module obtained by
applying $S^m_{C^{n}(L|K)}$ (symmetric tensor product of dg-modules over $C^{n}(L|K)$) simplex-wise. In other words: the cosimplicial (i.e.\@ simplex-wise) tensor horizontally and the graded tensor vertically. Using Dold-Kan and Eilenberg-Zilber this could be translated to bicosimplicial objects, but we will abstain from doing so. 
\end{PAR}

\begin{LEMMA}\label{LEMMAGALOIS}
Let $L|K$ be a Galois extension of characteristic 0 with Galois group $G$. We have a commutative diagram
\[ \xymatrix{
W'_{K} \otimes_{K} K[\Delta^{n}] \otimes_{K} C^{\bullet}(L|K)  \ar[r]^-{\sim} \ar[d]_{\mathrm{d}_A} & \Hom(G^n, W')^{n+1}  \ar[d]^{\iota, {}^{\sigma_1} \iota, \dots, {}^{\sigma_1 \dots \sigma_n} \iota} \\
V_K \otimes_K C^{\bullet}(L|K)  \ar[r]^-{\sim}  &  \Hom(G^n, V)
}\] 
where the horizontal isomorphisms are induced by the ones from Proposition~\ref{PROPGALOIS} and
the value of the right hand side vertical map is, more precisely, given for $(a_0, \dots, a_n)$ by
\[  (\sigma_1, \dots, \sigma_n) \mapsto \iota(a_0(\sigma_1, \dots, \sigma_n)) + \sum_{i=1}^n \sigma_1 \cdots \sigma_i \iota \sigma_i^{-1} \cdots \sigma_1^{-1} a_i(\sigma_1, \dots, \sigma_n) .  \]
\end{LEMMA}
\begin{proof}Left to the reader. \end{proof}

\begin{LEMMA}\label{LEMMACOHA}
Let $L|K$ be a field extension of characteristic 0. 
\[ H^n \left(\Tot(S^i_{C^{\bullet}(L|K)}A^{\bullet}(L|K)) \right) \cong \begin{cases} \Lambda^n_K V & n = i, \\ 0 & \text{otherwise},  \end{cases}   \]
and for $K=L$, the filtration induced by the columns filtration of $S^n_L A^{\bullet}(L)$ is the filtration (\ref{eqhodge}). 
\end{LEMMA}
\begin{proof}
We have by definition:
\[ (S^n_{C^{\bullet}(L|K)}A^{\bullet}(L|K))_i = S^{n-i}_{C^{\bullet}(L|K)} (A^{\bullet}(L|K)_0) \otimes_{C^{\bullet}(L|K)} \Lambda^{i}_{C^{\bullet}(L|K)} (A^{\bullet}(L|K)_1) . \]
Thus by Lemma~\ref{LEMMACONTRACTDELTA} the associated unnormalized cochain complexes are acyclic unless $i=n$. Furthermore, we have
\[ \Lambda^n_{C^{\bullet}(L|K)} (A^{\bullet}(L|K)_1) \cong \Lambda^n_K V_K \otimes_K C^{\bullet}(L|K).  \]
From the acyclicity of $K \to C^{\bullet}(L|K)$ (Proposition~\ref{PROPAMITSUR}) follows that the cohomology of the  associated unnormalized cochain complex is concentrated in degree 0 
and equal to $\Lambda^n_K V_K$. The row spectral sequence thus degenerates and gives the first assertion. 

Over a field (case $L=K$), we can use the Eilenberg-Zilber and K\"unneth theorems. Therefore $H^\bullet(S^n_L A^{\bullet}(L))$ is just 
$S^n_L H^{\bullet}(A^{\bullet}(L))$ together with its induced filtration. Hence it suffices to prove the statement for $n=1$. 
The $E^2$-page of the columns spectral sequences then is
\[ \xymatrix{ 0 & W' & 0   & \cdots \\
W & 0 &  0 & \cdots } \]
and thus the induced filtration is indeed (\ref{eqhodge1}). 
\end{proof}

\begin{BEM}
A word of caution: 
The analogous identification with the symmetric tensor product of dg-modules over the dg-algebra associated with $C^{\bullet}(L|K)$ works only in the derived sense. 
However, we will be interested in the {\em underived} tensor product over $C^{\bullet}(L|K)$ so we will never use this connection. 
\end{BEM}

We now relate the above abstract cosimplicial dg-modules with the  sheaf resolutions (Proposition~\ref{PROPRES}) and their associated filtrations $F_{\OOO_X}^{\bullet}$.
\begin{PROP}\label{PROPCOMP}
Let $X$ be a complex torus with Hodge structure $V_\Q = H^1(X, \Q)$, $W = H^1(X, \OOO_X)$, $W' = H^0(X, \Omega^1_X)$. 
Let $K \subseteq \C$ be a subfield. 

There is a  morphism of filtered vector spaces (functorial in $K$)
\[ H^n \left(\Tot(S^n_{C^{\bullet}(\C|K)} A^{\bullet}(\C|K)) \right) \to H^n(X, K),  \]
which is an isomorphism on the underlying vector spaces\footnote{not necessarily strict, i.e.\@ not necessarily an isomorphism of filtered vector spaces} for  
the filtration $F_{\OOO_X}^{\bullet}$ on the right hand side and the one induced by the
column filtration 
 on the left. 
It is strict for $K=\C$.
\end{PROP}

\begin{BEM}\label{REMCOMP}
We get a commutative diagram of filtered vector spaces: 
\[ \xymatrix{ 
H^n \left(\Tot(S^n_{C^{\bullet}(\C|\Q)} A^{\bullet}(\C|\Q)) \right)  \ar[r]^-{\sim} \ar@{^{(}->}[d] & H^n(X, \Q) \ar@{^{(}->}[d] \\
H^n \left(\Tot(S^n_{\C}  A^{\bullet}(\C)) \right)  \ar[r]^-{\sim}  & H^n(X, \C)
}  \]
where the bottom horizontal morphism is an isomorphism of filtered vector spaces. 
Hence the strictness of the left vertical morphism would imply the strictness of the right vertical morphism and then, a posteriori, also 
of the top horizontal morphism.
\end{BEM}

\begin{proof}[Proof of Proposition~\ref{PROPCOMP}]
We have a resolution
\[ \xymatrix{ 0 \ar[r] & \OOO_X \ar[r] & \mathcal{E}^{0,0}_X \ar[r]^{\overline{\partial}}  &  \mathcal{E}^{0,1}_X \ar[r]^{\overline{\partial}} &  \mathcal{E}^{0,2}_X \ar[r]^{\overline{\partial}}  & \cdots } \]
where $\mathcal{E}^{p,q}_X$ is the sheaf of $C^{\infty}$-differential forms on $X$ of degree $p,q$.
We thus obtain a double complex 
\begin{equation} \label{eqdc1} \xymatrix{  \mathcal{E}^{0,0}_X \ar[r]  \ar[d] & [T^2_K( \mathcal{E}_X^{0,\bullet})]^0 \ar[r] \ar[d]  & [T^3_K( \mathcal{E}_X^{0,\bullet})]^0 \ar[r]  \ar[d] & \\
 \mathcal{E}^{0,1}_X \ar[r]  \ar[d] & [T^2_K( \mathcal{E}_X^{0,\bullet})]^1 \ar[r] \ar[d]  & [T^3_K( \mathcal{E}_X^{0,\bullet})]^1 \ar[r] \ar[d]  &  
  \\
 \mathcal{E}^{0,2}_X \ar[r]  \ar[d] & [T^2_K( \mathcal{E}_X^{0,\bullet})]^2 \ar[r]  \ar[d] & [T^3_K( \mathcal{E}_X^{0,\bullet})]^2 \ar[r]  \ar[d] &  \\ \vdots & \vdots & \vdots } 
 \end{equation}
 which is a simultaneous resolution of $\OOO_X \to T^2_{K} \OOO_X \to T^3_{K} \OOO_X \to \cdots$.
 The double complex (\ref{eqdc1}) is the unnormalized cochain complex of a cosimplicial dg-algebra sheaf
\[ T_K^{\bullet} (\mathcal{E}^{0,\bullet}_X) :=  \left( \xymatrix{  \mathcal{E}^{0,\bullet}_X \ar@<3pt>[r]  \ar@<-3pt>[r]  \ar@{<.}[r]  &  T_K^2 (\mathcal{E}^{0,\bullet}_X) \ar@<-6pt>[r] \ar[r] \ar@<6pt>[r]  \ar@<-3pt>@{<.}[r]  \ar@<3pt>@{<.}[r]  \ar[r] & \cdots  } \right). \]
The sheaves $[T^i_K( \mathcal{E}^{0,\bullet}_X)]^k$ which are direct sums of tensor products of
\[   T_K^l \mathcal{E}^{0,m}_X \] 
are again flasque. Hence applying global sections to this
double complex yields a double complex whose columns filtration determines the  filtration $F_{\OOO_X}^{\bullet}$ on $H^n(X, K)$. 

Recall that we have isomorphisms
\[ X \cong \Gamma \backslash V_\R^* \cong \Gamma \backslash V^* / W^*,  \]
where $\Gamma \subset V_{\Q}^*$ is a suitable lattice and where $V_{\R}$ obtains its complex structure from the $\R$-isomorphism with $V^* / W^*$.

We have similar sheaves also on $V_{\R}$  (considered as complex manifold)
and a cosimplicial dg-algebra sheaf
\begin{equation} \label{eqdc2} 
T_K^{\bullet} (\mathcal{E}^{0,\bullet}_{V_{\R}}) :=  \left( \xymatrix{  \mathcal{E}^{0,\bullet}_{V_{\R}} \ar@<3pt>[r]  \ar@<-3pt>[r]  \ar@{<.}[r]  &  T_K^2 (\mathcal{E}^{0,\bullet}_{V_{\R}}) \ar@<-6pt>[r] \ar[r] \ar@<6pt>[r]  \ar@<-3pt>@{<.}[r]  \ar@<3pt>@{<.}[r]  \ar[r] & \cdots  } \right) 
\end{equation}
and taking global sections of (\ref{eqdc1}) is the same as taking global sections of (the unnormalized cochain complex of) (\ref{eqdc2}) followed by taking $\Gamma$-invariants. 

 We will now restrict our attention to the sub-algebra of {\em polynomial} anti-holomorphic forms  $S^{\bullet}_{\C} B \subset \mathcal{E}^{0,\bullet}(V_{\R})$. More precisely, 
consider the dg-module 
\[ B := \left( \vcenter{ \xymatrix{ V \ar[d]^{\mathrm{d}} \\ W }} \right) \]
with $V$ in degree 0 and $W$ in degree 1 and
consider the symmetric dg-algebra $S^{\bullet}_{\C} B$. 
On it the lattice $\Gamma \subset V_{\Q}^*$ acts by mapping $v \mapsto v + \lambda(v)$ where $\lambda(v) \in \C$ is considered to be in $S^0 B \cong \C$, and acts trivially on $W$.
We have a unique embedding 
\[ S^{\bullet}_{\C} B \rightarrow \mathcal{E}^{0,\bullet}(V_{\R}) \]
given by mapping $S^0_{\C} B = \C$ to constant functions, and defined on $B$ by mapping $v \in V_{\C}$ to $v$ considered as $\C$-valued function $V_{\R}^* \to \C$ and compatible with $\mathrm{d}$ in $B$ and $\overline{\partial}$ on 
$\mathcal{E}^{0,\bullet}(V_{\R})$. It other words, it maps an element in 
$W$ to the corresponding anti-holomorphic differential form. Note that $\mathcal{E}^{0,\bullet}$ is also a differential graded commutative algebra, so this association extends to $S^{\bullet}_{\C} B$.  
It is $\Gamma$-equivariant. 

This construction extends to a $\Gamma$-equivariant 
morphism of cosimplicial dg-modules
\[ \xymatrix{ K \ar[d]  \ar[r] & S_{\C}^{\bullet} B \ar[d]  \ar@<3pt>[r] \ar@<-3pt>[r] \ar@{<.}[r] & T_K^2 (S_{\C}^{\bullet} B) \ar[d]  \ar@<-6pt>[r] \ar[r] \ar@<6pt>[r] \ar@<-3pt>@{<.}[r] \ar@<3pt>@{<.}[r] \ar[r] & \cdots  \\
 K \ar[r] & \mathcal{E}^{0,\bullet}(V_{\R}) \ar@<3pt>[r]  \ar@<-3pt>[r]  \ar@{<.}[r]  &  T_K^2 (\mathcal{E}^{0,\bullet}(V_{\R})) \ar@<-6pt>[r] \ar[r] \ar@<6pt>[r]  \ar@<-3pt>@{<.}[r]  \ar@<3pt>@{<.}[r]  \ar[r] & \cdots  }\]

Proposition~\ref{PROPSYMMETRIC} leads to consider the cosimplicial dg-module (that we define here for an arbitrary field extension $L|K$):
\begin{equation}\label{eqB} B^{\bullet}(L|K) := \left( \vcenter{ \xymatrix{  V_K \otimes_K K[\Delta^{\bullet}] \otimes_K C^{\bullet}(L|K) \ar[d] \\
W_K \otimes_K K[\Delta^{\bullet}] \otimes_K C^{\bullet}(L|K)
}}\right) 
\end{equation}
where $W_K$ is any $K$-form of $W$, 
and it states that there is an isomorphism of cosimplicial dg-modules
\[ S_{C^{\bullet}(L|K)}^{\bullet} B^{\bullet}(L|K) \cong  T_K^{\bullet} (S_L^{\bullet} B^{\bullet}).   \]
 Thus, for $L=\C$, its composition with the embedding into $ T_K^{\bullet}  \mathcal{E}^{0,\bullet}(V_{\R})$ induces a morphism of spectral sequences from the columns spectral sequence associated with $(S_{C^{\bullet}(\C|K)}^{n} B^{\bullet}(\C|K))^{\Gamma}$ (for any $n$) to 
 the hypercohomology spectral sequence associated with the resolution (cf.\@ Proposition~\ref{PROPRES}) of $K_X$.
 
 The following Lemmata relate the invariants $(S^n_{C^{\bullet}(L|K)} B^{\bullet}(L|K))^{\Gamma}$ with the abstractly defined cosimplicial dg-modules
 $S^n_{C^{\bullet}(L|K)} A^{\bullet}(L|K)$ which proves the statement of the Proposition, except for the strictness in the case $L=K$. 
 For the latter, observe that the isomorphism 
 \[ H^n \left(\Tot(S^n_{\C} A^{\bullet} (\C)) \right) \to H^n \left( \Tot(T_K^{\bullet} (\mathcal{E}^{0,\bullet}(X))) \right)   \]
is strict because the filtration on both sides is the Hodge filtration by Lemma~\ref{LEMMACOHA}, and by Proposition~\ref{PROPL}, respectively.
 \end{proof}

Recall from \ref{PARCANCOSIMP} the definition of $K[\Delta^{\bullet}]_0$.
\begin{LEMMA} We have
\[ { \footnotesize
 \left( S^{\bullet}_{C^\bullet(L|K)} \left( \vcenter{ \xymatrix{  V_K \otimes_K K[\Delta^{\bullet}] \otimes_K C^\bullet(L|K)  \ar[d] \\
W_K \otimes_K K[\Delta^{\bullet}] \otimes_K C^\bullet(L|K)
} } \right) \right)^{\Gamma} 
= 
S^{\bullet}_{C^\bullet(L|K)}  \left( \vcenter{ \xymatrix{  V_K \otimes_K K[\Delta^{\bullet}]_0 \otimes_K C^\bullet(L|K)  \ar[d] \\
W_K \otimes_K K[\Delta^{\bullet}] \otimes_K C^\bullet(L|K)
} } \right) } \]
\end{LEMMA}
\begin{proof}
An element $\gamma \in \Gamma$ acts
trivially on $W_K \otimes_K K[\Delta^{\bullet}] \otimes_K C^\bullet(L|K)$ and
on $V_K \otimes_K K[\Delta^{\bullet}] \otimes_K C^\bullet(L|K)$ by mapping
\[v \otimes [i] \otimes l_0 \otimes \dots \otimes l_n\]
 to 
\[ v \otimes l_0 \otimes \dots \otimes l_n + l_0 \otimes \dots \otimes l_{i-1} \otimes l_i \gamma(v) \otimes l_{i+1} \otimes \cdots \otimes l_n  \]
\[ =  v \otimes l_0 \otimes \dots \otimes l_n  + \gamma(v) l_0 \otimes \dots \otimes l_n  \]
(using that $\gamma(v) \in K$), which is independent of $i$. Elements in the right hand side are thus invariant. 
Furthermore, the morphism $\gamma - \id$ shifts the filtration by tensor degree (not the dg-degree) by 1 and is {\em on graded pieces} given by the morphism
(abbreviating $C^\bullet(L|K)$ as $C^{\bullet}$) 
\begin{eqnarray*} 
 && \left( \Lambda^{n_1}_{C^\bullet}  W_K \otimes_K K[\Delta^\bullet] \otimes_K C^\bullet  \right)  \otimes
 \left( S^{n_2}_{C^\bullet}  V_K \otimes_K K[\Delta^\bullet] \otimes_K C^\bullet  \right) \\
&\to& 
 \left( \Lambda^{n_1}_{C^\bullet}  W_K \otimes_K K[\Delta^\bullet] \otimes_K C^\bullet  \right)  \otimes
\left( S^{n_2-1}_{C^\bullet}  V_K \otimes_K K[\Delta^\bullet] \otimes_K C^\bullet  \right)
 \end{eqnarray*}
obtained by extending the map
\[ V_K \otimes_K K[\Delta^n] \otimes_K C^n(L|K)   \to C^n(L|K)  \]
\[ v \otimes [i] \otimes l_0 \otimes \dots \otimes l_n \to \gamma(v) l_0 \otimes \dots \otimes l_n   \]
(independent of $i$)
as a derivation (in the sense of the symmetric algebra, not involving any degree). The set of maps $\gamma_i - 1$, where $\gamma_i$ runs through a basis of $\Gamma$, has thus joint kernel contained in
\[ \left( \Lambda^{n_1}_{C^\bullet(L|K)}  W_K \otimes_K K[\Delta^\bullet] \otimes_K L_\bullet \right)  \otimes
 \left( S^{n_2}_{C^\bullet(L|K)}  V_K \otimes_K K[\Delta^\bullet]_0 \otimes_K C^\bullet(L|K)  \right) \]
using that $K$ has characteristic zero. 
\end{proof}

 \begin{LEMMA}
 There is a chain of morphisms of cosimplicial dg-modules over $K$  (cf.\@ (\ref{eqB}) and Definition~\ref{KEYDEF})
  \[  (S^n_{C^{\bullet}(L|K)}B_{\bullet}(L|K))^{\Gamma}  \rightarrow \cdots \leftarrow S^n_{C^{\bullet}(L|K)} A_{\bullet}(L|K) \]
 that induce vertical quasi-isomorphisms between their double complexes of unnormalized cochains
and hence isomorphisms of columns spectral sequences from $E^1$ onwards. 
\end{LEMMA}
\begin{proof}
The chain is induced by the following maps
\begin{gather*} { \footnotesize
\left( \vcenter{ \xymatrix{  V_K \otimes_K K[\Delta^{\bullet}]_0 \otimes_K C^{\bullet}(L|K) \ar[d] \\
W_K \otimes_K K[\Delta^{\bullet}] \otimes_K C^{\bullet}(L|K)
}}\right) \rightarrow
\left( \vcenter{ \xymatrix{  V_K \otimes_K K[\Delta^{\bullet}] \otimes_K C^{\bullet}(L|K)  \ar[d] \\
(W_K \otimes_K K[\Delta^{\bullet}] \otimes_K C^{\bullet}(L|K)) \oplus (V_K \otimes_K C^{\bullet}(L|K))
} } \right) } \\
\leftarrow
{ \footnotesize \left( \vcenter{ \xymatrix{  W'_K \otimes_K K[\Delta^{\bullet}] \otimes_K C^{\bullet}(L|K)  \ar[d] \\
V_K \otimes_K C^{\bullet}(L|K)
} } \right)  }
\end{gather*}
which  obviously induce quasi-isomorphisms of the columns. 
Hence they induce  quasi-isomorphisms of the columns after applying $S^{n}_{C^{\bullet}(L|K)}$ \cite[Tag 064K]{SP} because they consist of free (hence flat) $C^{n}(L|K)$-modules. Hence they induce a filtered quasi-isomorphism after applying  $\Tot(S^{\bullet}_{C^{\bullet}(L|K)} - )$, and in fact, isomorphisms of the whole column spectral sequences from page $E^1$ onwards. 
\end{proof}

\section{Proof of strictness}

Recall Definition~\ref{KEYDEF}. Applying $S^{2n}$ to its functoriality in $K$ induces the following projection with kernel $\ker^{2n}$, i.e.\ induces a short exact sequence of cosimplicial dg-modules of the form:
\begin{equation}\label{eqex} \xymatrix{ 0 \ar[r] &  \ker^{2n}  \ar[r] &  S^{2n}_{C^{\bullet}(L|K)} A^{\bullet}(L|K)  \ar[r] &  S^{2n}_{L} A^{\bullet}(L)  \ar[r] &  0. }
\end{equation}
\comment{
Inserting the definition: 
\[
0 \to \ker^n \to 
S^n_{C^{\bullet}(L|K)} \left( \vcenter{ \xymatrix{  W'_K \otimes_K K[\Delta^{\bullet}] \otimes_K C^{\bullet}(L|K)  \ar[d] \\
V_K \otimes_K C^{\bullet}(L|K)
} } \right)  \to 
S^n_{L} \left( \vcenter{ \xymatrix{  W'_L  \otimes_L L[\Delta^{\bullet}]   \ar[d] \\
V_L
} } \right) \to 0
\]
}
This section is dedicated to the question under which conditions the projection map induces a strict morphism of filtered vector spaces on the total cohomology of the double complex of unnormalized cochains (always w.r.t.\@ the column filtration).
For $L = \C$ and $K = \Q$ the strictness for the $n$-th filtration step  implies the surjectivity of 
\[ \mathrm{dlog}: H^{n}(X, \mathcal{K}_n^{M,\mathrm{an}})_{\Q}  \to H^{n,n} \cap H^{2n}(X, \Q)   \]
(for a complex torus $X$ described by the Hodge structure used to define $A^{\bullet}(\C|\Q)$) 
by Propositions~\ref{PROPL}, \ref{PROPK}, and \ref{PROPCOMP} (cf.\@ also Remark~\ref{REMCOMP}) because $\mathcal{K}_n^{M,\mathrm{an}}$ is a quotient of
$\Lambda^n_{\Z} \OOO_X^*$. To begin with, one can w.l.o.g.\@ take $L$ to be a field of definition for $\mathrm{d}: V \to W$:

\begin{PROP}\label{PROPFOD}
For fields $L|L'|K$, we have isomorphisms of filtered vector spaces
\begin{equation}\label{eqprop1}
 H^i \left(\Tot(S^j_{C^{\bullet}(L|K)} A^{\bullet}(L|K)) \right) \cong H^i \left(\Tot(S^j _{C^{\bullet}(L'|K)} A^{\bullet}(L'|K)) \right)
 \end{equation}
if $\mathrm{d}$ is defined over $L'$ and
\begin{equation}\label{eqprop2}
 H^i \left( \Tot(S^j_{C^{\bullet}(L|L')} A^{\bullet}(L|L')) \right) \cong H^i \left( \Tot(S^j_{C^{\bullet}(L|K)} A^{\bullet}(L|K)) \right) \otimes_K L'   
 \end{equation}
if $\mathrm{d}$ is defined over $K$.
\end{PROP}

\begin{proof}
The existence of the isomorphism (\ref{eqprop1}) follows from Proposition~\ref{PROPAMITSUR} because we have $A^{\bullet}_{L|K} \cong A^{\bullet}_{L'|K} \otimes_{C^{\bullet}(L'|K)} C^{\bullet}(L|K)$ if $\mathrm{d}$ is defined over $L'$
and the homotopies constructed in Proposition~\ref{PROPAMITSUR} are obviously compatible with the differentials in the $A^{\bullet}$.
For isomorphism (\ref{eqprop2}) we may assume $L=L'$, and $L=K$, respectively, by (\ref{eqprop1}). Under these assumptions, we have
\[ A^{\bullet}(L') \cong A^{\bullet}(K) \otimes_K L' \]
and the statement is obvious. 
\end{proof}

By (\ref{eqprop2}) the map 
\begin{equation} \label{eqfunct} H^{2n}(\Tot(S^{2n}_{C^{\bullet}(L|K)} A^{\bullet}(L|K))) \to H^{2n}(\Tot(S^{2n}_L A^{\bullet}(L)))  \end{equation}
induced by the projection map in (\ref{eqex}), is automatically strict if $\mathrm{d}$ is defined over $K$. 
\begin{FRAGE}
Is the map (\ref{eqfunct}) always strict?
\end{FRAGE}

\begin{PAR}\label{PAREINFTY}
From the commutative diagram with exact rows (where the terms $E^{\infty}$ are the abutments of the column spectral sequences)
\[ { \footnotesize \xymatrix@=1em{
0 \ar[r] & F^{i+1} H^{2n}(\Tot(S^{2n} A^{\bullet}(L|K))) \ar[r] \ar[d] & F^{i} H^{2n} (\Tot(S^{2n} A^{\bullet}(L|K))) \ar[r] \ar[d] & E^{\infty}_{i,2n-i}(S^{2n} A^{\bullet}(L|K)) \ar[d] \ar[r] & 0 \\
0 \ar[r] & F^{i+1} H^{2n}(\Tot(S^{2n} A^{\bullet}(L))) \ar[r] & F^{i} H^{2n}(\Tot(S^{2n} A^{\bullet}(L))) \ar[r] & E^{\infty}_{i,2n-i}(S^{2n} A^{\bullet}(L)) 
 \ar[r] & 0
 } }\]
follows by induction that the strictness of the map (\ref{eqfunct}) for the $n$-th filtration step follows from the injectivity of the induced map
\[ E^{\infty}_{i,2n-i}(S^{2n}_{C^{\bullet}(L|K)} A^{\bullet}(L|K)) \to E^{\infty}_{i,2n-i}(S^{2n}_L A^{\bullet}(L))  \]
 for all $0 \le i<n$. 
 \end{PAR}
So far, we will be able to show only (cf.\@ page~\pageref{KEYPROPPROOF} for the proof): 
\begin{PROP}\label{KEYPROP}
If $L|K$ is Galois then the map\footnote{the evenness of the dimension plays no role, hence we replace $2n$ by $n$} 
\[ E^{\infty}_{i,n-i}(S^n_{C^{\bullet}(L|K)} A^{\bullet}(L|K)) \to E^{\infty}_{i,n-i}(S^n_L A^{\bullet}(L))  \]
induced by the projection map in (\ref{eqex}) is injective for $i=0,1$.
\end{PROP}

Proposition~\ref{KEYPROP} is sufficient to prove the holomorphic $(2,2)$-theorem {\em for complex tori with algebraic (normalized) period matrix}, in particular for Abelian varieties of CM-type, i.e.\@ Theorem~\ref{SATZ22}:
\begin{proof}[Proof of Theorem~\ref{SATZ22}.] \label{SATZ22PROOF}
By Propositions~\ref{PROPL}, \ref{PROPK}, and \ref{PROPCOMP} (cf.\@ also Remark~\ref{REMCOMP}) the surjectivity of
\[ \mathrm{dlog}: H^{2}(X, \mathcal{K}_2^{M,\mathrm{an}})_{\Q}  \to H^{2,2} \cap H^{4}(X, \Q)   \]
follows from the strictness for the second filtration step of the map (\ref{eqfunct}) for $L= \C$ and $K=\Q$ because $\mathcal{K}_2^{M,\mathrm{an}}$ is a quotient of
$\Lambda^2_{\Z} \OOO_X^*$. 
By Proposition~\ref{PROPFOD} we are reduced to prove the strictness of (\ref{eqfunct}) for the second filtration step for $L$ being a field of definition of $\mathrm{d}: V = H^1(X, \C) \to W = H^1(X, \OOO_X)$ (w.r.t.\@ the fixed $\Q$-structure $V_{\Q} = H^1(X, \Q)$ on $V$). 
The algebraicity of the normalized period matrix implies that $\mathrm{d}$ is defined over a finite extension of $\Q$ which w.l.o.g.\@ may assumed to be Galois. 
The discussion in \ref{PAREINFTY} shows that Proposition~\ref{KEYPROP} implies the strictness of  (\ref{eqfunct}) in this case. 
\end{proof}

\begin{PAR}\label{PARKEYPROPT}
The symmetrization map $S$ (in the graded sense) and the projection $P$ induce morphisms of exact sequences of cosimplicial dg-modules
\begin{equation*} 
\vcenter{ \xymatrix{
0 \ar[r] &  \ker^n \ar[r] \ar@<2pt>[d]^S &  S^n_{C^{\bullet}(L|K)} A^{\bullet}(L|K) \ar[r] \ar@<2pt>[d]^S  &  S^n_{L} A^{\bullet}(L) \ar[r] \ar@<2pt>[d]^S  &  0  \\
0 \ar[r] &  \widetilde{\ker}^n \ar[r] \ar@<2pt>[u]^P  &  T^n_{C^{\bullet}(L|K)} A^{\bullet}(L|K) \ar[r] \ar@<2pt>[u]^P  &  T^n_{L} A^{\bullet}(L) \ar[r] \ar@<2pt>[u]^P &  0  \\
}}
\end{equation*}
such that $PS = \id$. This makes the sequence of spectral sequences for $S^n$ a direct summand of the one for $T^n$ and it suffices to see the statement of Proposition~\ref{KEYPROP} for $T^n$.
\end{PAR}

\begin{LEMMA}\label{LEMMASPLIT}
If $L|K$ is Galois, 
the exact sequence
\begin{equation}\label{eqex2} 
\vcenter{ \xymatrix{
0 \ar[r] &  \widetilde{\ker}^n \ar[r]  &  T^n_{C^{\bullet}(L|K)} A^{\bullet}(L|K) \ar[r]   &  T^n_{L} A^{\bullet}(L) \ar[r]  &  0  \\
}}
\end{equation}
 induces 
\begin{enumerate}
\item \dots {\em split} exact sequences (for the column spectral sequences)
\[ \vcenter{ \xymatrix{
 0 \ar[r] & E^0(\widetilde{\ker}^n) \ar[r] & E^0(T^n_{C^{\bullet}(L|K)} A^{\bullet}(L|K)) \ar[r] & E^0( T^n_L A^{\bullet}(L)) \ar[r] &  0, 
 }} \]
(that is, sequence and splitting compatible with the differentials $\mathrm{d}^0$ --- which are just the vertical differentials in the double complexes),
\item \dots exact sequences
\begin{equation*} \xymatrix{
 0 \ar[r] & E^1(\widetilde{\ker}^n) \ar[r] & E^1(T^n_{C^{\bullet}(L|K)} A^{\bullet}(L|K)) \ar[r] & E^1(T^n_L A^{\bullet}(L)) \ar[r] & 0, 
 } \end{equation*}
\item \dots and exact sequences
\begin{equation} \label{eqe2}   \xymatrix{
E^2(\widetilde{\ker}^n) \ar[r] & E^2(T^n_{C^{\bullet}(L|K)} A^{\bullet}(L|K)) \ar[r] & E^2(T^n_L A^{\bullet}(L)), 
}  \end{equation}
\end{enumerate}
but not necessarily further. 
\end{LEMMA}

\begin{proof}
1.\@ In the Galois case, by Lemma~\ref{LEMMAGALOIS}, the projection in the exact sequence (\ref{eqex2}) restricted to the columns and for $n=1$ looks like 
\[ \xymatrix{
 \Hom(G^j, W')^{j+1}  \ar[d]_{\iota, {}^{\sigma_1} \iota, \dots, {}^{\sigma_1 \cdots \sigma_j} \iota} \ar[r] & (W')^{j+1} \ar[d]^{\iota, \dots, \iota} \\ 
 \Hom(G^j, V) \ar[r] & V
}\] 
where the horizontal maps are evaluation at $(1, \dots, 1)$. These have obviously sections $s$, given by the delta function at $(1, \dots, 1)$, which are compatible with the differential.
The required splitting is then given by $x_0 \otimes_L \cdots \otimes_L x_n \mapsto s(x_0) \otimes_{C^{j}(L|K) } \cdots  \otimes_{C^{j}(L|K) }  s(x_n)$.
This respects the relations of $\otimes_L$ because the section is $L$-linear w.r.t.\@ any of the $j+1$ $L$-module structures on $C^{j}(L|K)$.
2.\@ and 3.\@ follow immediately from 1.
\end{proof}

\begin{KEYLEMMA}\label{KEYLEMMA}
The morphism
\[ \mathrm{d}^1: E^1_{1, j}(\widetilde{\ker}^n) \to E^1_{2,j}(\widetilde{\ker}^n)  \]
(for the column spectral sequence) is injective if $j  < n$.
\end{KEYLEMMA}

\begin{proof}
Recall that we have
\[ A^{\bullet}(L|K)_0 = W'_K \otimes_K K[\Delta^{\bullet}] \otimes_K C^{\bullet}(L|K) \qquad A^{\bullet}(L|K)_1 = V_K \otimes_K C^{\bullet}(L|K) \]
which we will denote here just by $A_0^{\bullet}$, and $A_1^{\bullet}$, respectively.
The cosimplicial modules
\[ A_{j_1}^{\bullet} \otimes_{C^{\bullet}(L|K)} \cdots \otimes_{C^{\bullet}(L|K)} A_{j_n}^{\bullet} \]
with $j_1 + \cdots + j_n = j$ and $j_k \in \{0,1\}$ are all contractible for $j < n$ in multiple ways via 
\[ \kappa_l^K: a_1 \otimes \cdots \otimes a_n \mapsto  \sum_{k=1}^i (-1)^{k-1} s_k(a_1) \otimes \cdots \otimes s_{k}(a_{l-1}) \otimes [k] \otimes s_k(a_{l,k}) \otimes s_k(a_{l+1}) \otimes \cdots \otimes   s_k(a_n) \]
Here the $a_{l,k} \in W'_K \otimes_K C^{i}(L|K)$ are determined by $a_l = \sum_{k=0}^i [k] \otimes a_{l,k}$, and
$l$ {\em can be any index such that $j_l = 0$}.
This is because by Lemma~\ref{LEMMACONTRACTDELTA} we have
\[ \kappa_l^K \mathrm{d} +  \mathrm{d} \kappa_l^K = \id,  \]
where $\mathrm{d}$ is the horizontal differential in the double complexes of unnormalized cochains. 
Furthermore, we have
\begin{equation} \label{eqcomm} \kappa_l^K \mathrm{d}_A^{(l')} = \mathrm{d}_A^{(l')} \kappa_{l}^K   \end{equation}
for $l' \not = l$ where $\mathrm{d}_A^{(l')} := 1 \otimes \cdots \otimes 1 \otimes \mathrm{d}_A \otimes 1 \otimes \cdots \otimes 1$ (at position $l'$) because $\mathrm{d}_A$ is a morphism of cosimplicial $C^{\bullet}(L|K)$-modules and thus commutes with all codegeneracies. 
We also have for field homomorphisms $\iota: K \to K'$ that
\[ \iota \kappa_l^K  =  \kappa_l^{K'} \iota,   \]
where $\iota$ denotes, by abuse of notation, the map
\[ A_{j_1}^{\bullet}(L|K) \otimes_{C^{\bullet}(L|K)} \cdots \otimes_{C^{\bullet}(L|K)} A_{j_n}^{\bullet}(L|K)  \to A_{j_1}^{\bullet}(L|K') \otimes_{C^{\bullet}(L|K')} \cdots \otimes_{C^{\bullet}(L|K')} A_{j_n}^{\bullet}(L|K') \]
induced by $\iota$. In particular, $\kappa_l$ maps components of $\widetilde{\ker}^n$ to components of  $\widetilde{\ker}^n$. 

After this preparation we are able to show that 
\begin{equation*} \mathrm{d}^1: E^1_{1, j}(\widetilde{\ker}^n) \to E^1_{2,j}(\widetilde{\ker}^n)   \end{equation*}
is injective. Indeed, let $[y] \in E^1_{1, j}(\widetilde{\ker}^n)$ be represented by an element $y \in \bigoplus  A_{j_1}^1 \otimes \cdots \otimes A_{j_n}^1$ with components $y_{j_1, \dots, j_n}$.
Since $j < n$ by assumption, we can choose $l$ with $j_l = 0$ and write
\[ y_{j_1, \dots, j_n} = \kappa_l^K \mathrm{d} y_{j_1, \dots, j_n}   + \mathrm{d} \underbrace{ \kappa_l^K y_{j_1, \dots, j_n} }_{=0}    \]
(Notice that $\widetilde{\ker}_{0,j} = 0$). 
Now assume that $[y]$ is in the kernel of $\mathrm{d}^1$. Then we have 
\[ \mathrm{d} y_{j_1, \dots, j_n} = \sum_{l'\ |\ j_{l'}=1} \pm \mathrm{d}_A^{(l')} z_{j_1, \dots, j_{l'-1}, 0,  j_{l'+1}, \dots, j_n}. \]
for a suitable collection $z_{j_1', \dots, j_n'} \in A_{j_1'}^2 \otimes \cdots \otimes A_{j_n'}^2$ running over $j_1' + \cdots + j_n' = j-1$.
 Thus for any $j_1, \dots, j_n$ we have using (\ref{eqcomm})
\begin{gather*} y_{j_1, \dots, j_n} = \kappa_l^K \sum_{l'\ |\ j_{l'}=1} \pm \mathrm{d}^{(l')}_A z_{j_1, \dots, j_{l'-1}, 0,  j_{l'+1}, \dots, j_n} \\
 = \sum_{l'\ |\ j_{l'}=1 } \pm \mathrm{d}^{(l')}_A  \kappa_l^K  z_{j_1, \dots, j_{l'-1}, 0,  j_{l'+1}, \dots, j_n} .
 \end{gather*}
Thus by Lemma~\ref{LEMMA1} below there exists $\widetilde{z} \in   (T^n_{C^{1}(L|K)} A^1)_{j-1}$ such that $\mathrm{d}_A \widetilde{z} = y$. This means that $[y]$ is in the kernel of 
\[ E^1_{1, j}(\widetilde{\ker}^n)  \to E^1_{1, j}(T^n_{C^{\bullet}(L|K)} A^{\bullet}(L|K)) \]
which is zero by Lemma~\ref{LEMMASPLIT}.
\end{proof}

\begin{PAR} \label{PARDECOMP} In the proof of Lemma~\ref{KEYLEMMA} the following linear algebra trick was used:
Let $K$ be a field and
let $\gamma: A_0 \to A_1$ be a homomorphism of finite dimensional vector spaces. 
Choosing complements, we have $A_0 = H_0 \oplus A_0'$ and $A_1 = A_1' \oplus H_1$ such that $\gamma$ is induced by an isomorphism $A_0' \cong A_1'$.
Consider $A:=[A_0 \to A_1]$ as a dg-module with $\mathrm{d} = \gamma$ and consider the dg-module $T^n_K A$.
We have by definition
\[ (T^n_K A)_j = \sum_{j = j_1 + \cdots + j_n} A_{j_1} \otimes \cdots \otimes A_{j_n} \]
where $j_k \in \{0,1\}$ for all $k$.
\end{PAR}

\begin{LEMMA}\label{LEMMA1}
Let $R$ be a finite product of fields and
let $\gamma: A_0 \to A_1$ be a homomorphism of finitely generated $R$-modules. 
Consider $A:=[A_0 \to A_1]$ with $\mathrm{d} = \gamma$ as a dg-$R$-module and consider the dg-module $T^n_R A$.
Let $x \in (T^n_R A)_j$ be such that $\mathrm{d} x = 0$. Assume that for each component
\[ x_{j_1, \dots, j_n} \in A_{j_1} \otimes \cdots \otimes A_{j_n} \]
there is $y' \in (T^n_R A)_{j-1}$ such that in $A_{j_1} \otimes \cdots \otimes A_{j_n}$ we have
\begin{equation}\label{eqcond} x_{j_1, \dots, j_n} = (\mathrm{d} y')_{j_1, \dots, j_n} . \end{equation}
Then there exists an $y \in (T^n_R A)_{j-1}$ such that
$x = \mathrm{d} y$.
\end{LEMMA}
The point is, that the $y'$ may well depend on $j_1, \dots, j_n$. Obviously, only the components $y'_{j_1', \dots, j_n'}$ such that
$j_l' \le j_l$ for all $l$ matter. 
\begin{proof}
It obviously suffices to see this when $R$ is a field itself. 
By the K\"unneth theorem, we may assume, adding boundaries, that 
\[ x_{j_1, \dots, j_n} \in H_{j_1} \otimes \cdots \otimes H_{j_n} \]
for the $H_k$ chosen in \ref{PARDECOMP}.
Decomposing the $A_{j_1} \otimes \cdots \otimes A_{j_n}$ and $A_{k_1} \otimes \cdots \otimes A_{k_n}$ using the decompositions in \ref{PARDECOMP}, we see that 
the summand $H_{j_1} \otimes \cdots \otimes H_{j_n}$ is not in the image of any $1 \otimes \cdots \otimes 1 \otimes \gamma \otimes 1 \otimes \cdots \otimes 1$. Condition (\ref{eqcond}) thus implies $x_{j_1, \dots, j_n} = 0$. 
\end{proof}
\begin{BEM}
The statement of Lemma~\ref{LEMMA1} would be false for rings $R$ that admit finitely generated modules with non-trivial $\mathrm{Tor}$ groups. In particular, it is false for $C^n(L|K)$ when $n \ge 1$ and $L|K$ is transcendental. 
\end{BEM}

\begin{proof}[Proof of Proposition~\ref{KEYPROP}] \label{KEYPROPPROOF}
For $i=0,1$ there is a commutative diagram 
\[
\xymatrix{
E^{\infty}_{i,n-i}(T^n_{C^{\bullet}(L|K)} A^{\bullet}(L|K))  \ar@{^{(}->}[r] \ar[d] & \cdots \ar@{^{(}->}[r] &  E^2_{i,n-i}(T^n_{C^{\bullet}(L|K)} A^{\bullet}(L|K)) \ar@{^{(}->}[d] \\
E^{\infty}_{i,n-i}(T^n_{L} A^{\bullet}(L))  \ar@{^{(}->}[r] & \cdots \ar@{^{(}->}[r] &   E^2_{i,n-i}(T^n_L A^{\bullet}(L))
}
\]
where all the horizontal maps are injective because there are no differentials $\mathrm{d}^k$ for $k \ge 2$ with target $E^k_{i, n-i}$, and where the right-most vertical map is injective because of the exactness of (\ref{eqe2}) and because Lemma~\ref{KEYLEMMA} implies that $E^{2}_{1, n-1}(\widetilde{\ker}^n) = 0$ (and $E^0_{0,n}(\widetilde{\ker}^n) = 0$ trivially).
Thus also all vertical maps are injective which gives the statement by the discussion in \ref{PARKEYPROPT}.
\end{proof}

\newpage

\bibliographystyle{abbrvnat}
\bibliography{hodge}

\end{document}